\documentclass[12pt]{amsart}
\usepackage[shortlabels]{enumitem}
\usepackage{hyperref}
\usepackage{xcolor}
\usepackage{amsmath,amssymb,amsthm,amscd}
\newtheorem{theorem}{Theorem}

\newtheorem{corollary}[theorem]{Corollary}
\newtheorem{proposition}[theorem]{Proposition}
\newtheorem*{theorem*}{Theorem}

\theoremstyle{definition}
\newtheorem{remark}{Remark}
\usepackage[margin=3cm]{geometry}
\newcommand{\Z}{\mathbb{Z}}
\newcommand{\Q}{\mathbb{Q}}
\newcommand{\Gal}{\textrm{Gal}}
\newcommand{\oo}{\mathcal{O}}

\newcommand{\GL}{\textrm{GL}}
\newcommand{\tors}{\textrm{tors}}

\newcommand{\cI}{\mathcal{I}}
\newcommand{\cF}{\mathcal{F}}
\newcommand{\cE}{\mathcal{E}}
\newcommand{\cQ}{\mathcal{Q}}
\usepackage{mathrsfs}
\usepackage{url}
\date{}
\author{Tyler Genao}
\thanks{Email: \texttt{tylergenao@uga.edu}. 
This material is based upon work supported by the National Science Foundation Graduate Research Fellowship under Grant No. 1842396. Partial support was also provided by the Research and Training Group grant DMS-1344994 funded by the National Science Foundation.}
\title[Typically bounding torsion isogenous to rational $j$-invariant]{Typically Bounding Torsion on Elliptic Curves Isogenous to Rational $j$-Invariant}
\begin{document}
\begin{abstract}
We prove that the family $\cI_{F_0}$ of elliptic curves over number fields that are geometrically isogenous to an elliptic curve with $F_0$-rational $j$-invariant is typically bounded in torsion. Under an additional uniformity assumption, we also prove that the family $\cI_{d_0}$ of elliptic curves over number fields that are geometrically isogenous to an elliptic curve with degree $d_0$ $j$-invariant is typically bounded in torsion.
\end{abstract}
\maketitle
\section{Introduction}
By Merel's strong uniform boundedness theorem \cite[Corollaire]{Mer96}, for each integer $d\in\Z^+$ there exists a bound $B(d)$ on the size of torsion subgroups of elliptic curves over degree $d$ number fields. Whereas it is an open problem to determine sharp upper bounds for $d>3$, recent progress has been made towards understanding the asymptotic behavior of $B(d)$ when restricted to special families of elliptic curves.

Let us define a \textit{family} $\cF$ to be a collection of pairs $E_{/F}$ where $F$ is a number field and $E$ is an elliptic curve defined over $F$. 
Let us also recall that any subset $S\subseteq \Z^+$ has a well-defined upper asymptotic density
\[
\overline{\delta}(S):=\limsup_{x\rightarrow \infty} \frac{\#(S\cap [1,x])}{x}.
\]
Then following Clark, Milosevic and Pollack \cite{CMP18}, we say that a family $\cF$ is \textit{typically bounded in torsion} if for all $\epsilon>0$ there exists a constant $B(\epsilon)>0$ such that the set of ``bad degrees,"
\[
\mathcal{S}(\cF, B(\epsilon)):=\lbrace d\in\Z^+: \exists E_{/F}\in \cF~\textrm{with }[F:\Q]=d,\#E(F)[\tors]\geq B(\epsilon)\rbrace,
\]
has upper density at most $\epsilon$. Stated differently, $\cF$ is typically bounded in torsion if the torsion subgroups $E(F)[\tors]$ of elliptic curves $E_{/F}\in \cF$ can be absolutely bounded after removing from $\cF$ the elliptic curves $E_{/F}$ whose number field degrees $[F:\Q]$ lie in a subset of $\Z^+$ of arbitrarily small upper density.
For example, Bourdon, Clark and Pollack have shown that the family of all CM elliptic curves is typically bounded in torsion \cite[Theorem 1.1.i]{BCP17}.

The family of all elliptic curves over number fields is not typically bounded in torsion \cite[Theorem 1.7]{CMP18}. But for any number field $F_0$, the family of elliptic curves with rational $j$-invariant,
\[
\cE_{F_0}:=\lbrace E_{/F}:j(E)\in F_0\rbrace,
\]
is typically bounded in torsion: this was first proven conditional on GRH and with a technical hypothesis on $F_0$ \cite[Theorem 1.8]{CMP18}, and then unconditionally with no such hypothesis \cite[Theorem 1]{Gen22}. 

We will generalize this result. Let us define a superfamily $\cI_{F_0}$ of $\cE_{F_0}$,
\[
\cI_{F_0}:=\lbrace E_{/F}:E~\textrm{is geometrically isogenous to some }E'~\textrm{with }j(E')\in F_0\rbrace.
\]
Our main result is  the following.
\begin{theorem}\label{I_F is TBT}
For each number field $F_0$, the family $\cI_{F_0}$ is typically bounded in torsion.
\end{theorem}
The family $\cI_{F_0}$ is significantly larger than $\cE_{F_0}$, as we now explain. Given an algebraic number $j\in \overline{\Q}$ and any elliptic curve $E$ with $j$-invariant $j(E)=j$, the isogeny class of $E$ contains $j$-invariants $j'$ of arbitrarily large degree $[\Q(j'):\Q]$. In the case where $E$ has complex multiplication (henceforth abbreviated as CM) this follows from class field theory: given a CM elliptic curve $E$ whose geometric endomorphism ring is (isomorphic to) an imaginary quadratic order $\oo$, its $j$-invariant has degree $[\Q(j(E)):\Q]=\#\textrm{Pic}(\oo)$, where $\textrm{Pic}(\oo)$ is the class group of $\oo$. By a classic result of Heilbronn \cite{Hei34}, as the discriminant $\Delta(\oo)\rightarrow -\infty$ the class number $\#\textrm{Pic}(\oo)\rightarrow \infty$. Since any two CM elliptic curves will be geometrically isogenous if their geometric endomorphism rings have isomorphic fraction fields, one concludes that the isogeny class of any CM elliptic curve will have $j$-invariants of arbitrarily large degree.
On the other hand, given any non-CM elliptic curve $E$ defined over a number field $F_0$, Serre's open image theorem \cite[Th\'eor\`eme 2]{Ser72} tells us that for $\ell\gg_{E, F_0}0$ one has that the mod-$\ell$ Galois representation of $E$ is surjective. It follows then that for any order $\ell$ subgroup $C_\ell\subseteq E$, the field of definition of $C_\ell$ over $F_0$ has degree $\ell+1$. By \cite[Proposition 3.3]{Cla} the field of definition equals $F_0(j_\ell)$ where $j_\ell$ is the $j$-invariant of the quotient elliptic curve $E/C_\ell$. Thus the isogeny class of $E$ contains $j$-invariants of degree $\ell+1$ over $F_0$ for all primes $\ell\gg_{E,F_0} 0$.

The study of torsion subgroups from $\cI_{F_0}$ fits into a larger program of studying Galois representations which ``come from" elliptic curves over $F_0$. Let us define the family of \textit{$F_0$-curves},
\[
\cQ_{F_0}:=\lbrace E_{/F}:\forall \sigma\in G_{F_0}, ~E~\textrm{is geometrically isogenous to }E^\sigma\rbrace.
\]
Like the family $\cI_{F_0}$, the family $\cQ_{F_0}$ is closed under geometric isogeny. 
In fact, we have the containments
\[
\cE_{F_0}\subseteq \cI_{F_0}\subseteq \cQ_{F_0}.
\]
If $F_0=\Q$, then $\cQ_{F_0}=\cQ_\Q$ is the well-studied family of $\Q$-curves. By results of \cite{Rib92, KW09a, KW09b}, an elliptic curve is a $\Q$-curve iff it is a modular elliptic curve, i.e., is a quotient of the Jacobian $J_1(N)$ of the modular curve $X_1(N)$ for some $N\in\Z^+$. 
Furthermore, for any non-CM $\Q$-curve $E_{/F}$, if $[F:\Q]$ is odd then $E$ is isogenous over $F$ to an elliptic curve with $\Q$-rational $j$-invariant \cite[Theorem 2.7]{CN}. Additionally, this forces the prime divisors of $\#E(F)[\tors]$ to lie in the set $\lbrace 2,3,5,7,11,13\rbrace$, and in fact $\#E(F)[\tors]\leq 1441440\sqrt{35}\cdot \sqrt{[F:\Q]}$ -- these are consequences of recent work of Bourdon and Najman \cite[Proposition 4.1]{BN}. These results on torsion subgroups from $\cI_\Q$ are part of the motivation for studying the family $\cI_{F_0}$ in this paper. 

It is very interesting to ask whether $\cQ_{\Q}$ is typically bounded in torsion. Proving this using properties \textbf{P1} and \textbf{P2} (defined in Section \ref{SectionTBT}) requires at least 
knowledge of degrees of torsion points on \textit{central} $\Q$-curves (these are defined in e.g. \cite{CN}).
Alternatively, a certain uniform boundedness conjecture for rational non-cuspidal non-CM points on the full Atkin-Lehner quotient modular curves $X^*(N)$, as described by Elkies \cite[Section 3]{Elk04} (see also Ellenberg \cite[Conjecture 3]{Ell04}), would imply that $\cQ_\Q$ is typically bounded in torsion. But a proof of this boundedness conjecture currently seems out of reach. 	

\subsection{An additional result}
Similar to \cite[Theorem 2]{Gen22} we have an additional result under an extra hypothesis. For each integer $d_0\in\Z^+$, let us define the family
\[
\cI_{d_0}:=\lbrace E_{/F}:E~\textrm{is geometrically isogenous to some }E'~\textrm{with }[\Q(j(E')):\Q]=d_0\rbrace.
\]
A result of Larson and Vaintrob \cite[Theorem 1]{LV14} says that for any number field $F_0$ and for all primes $\ell\gg_{F_0} 0$, if an elliptic curve defined over $F_0$ has an $F_0$-rational isogeny of degree $\ell$, then the twelfth power of its isogeny character is either the twelfth power of an isogeny character from a CM elliptic curve, or the sixth power of the mod-$\ell$ cyclotomic character. In \cite[Theorem 7.9]{LV14} they give an upper bound on the implied constant $\ell\gg_{F_0} 0$ which depends on $F_0$. 
For $d_0\in \Z^+$, we let \textbf{LV}$(d_0)$ be the hypothesis that the implied constant $\ell\gg_{F_0} 0$ for \cite[Theorem 1]{LV14} can be chosen to be the same between any degree $d_0$ number field $F_0$. Then our proof of Theorem \ref{I_F is TBT} also proves the following.
\begin{theorem}\label{I_d is TBT}
For any integer $d_0\in \Z^+$, if $\textbf{\emph{LV}}(d_0)$ is true then the family $\cI_{d_0}$ is typically bounded in torsion.
\end{theorem}
\subsection{Notations and conventions}
Once and for all, let us fix an algebraic closure $\overline{\Q}$ of $\Q$. Then for each number field $F$, we let $G_F:=\Gal(\overline{\Q}/F)$ be its absolute Galois group. Unless stated otherwise, 
all elliptic curves are defined over number fields.

Without qualification, an isogeny of elliptic curves is assumed to be a geometric isogeny, i.e., a $\overline{\Q}$-rational isogeny; as an adjective, ``geometric" will mean $\overline{\Q}$-rational. We will use $E_{/F}$ to denote an elliptic curve $E$ defined over a number field $F$.
For any integer $N\in\Z^+$, we will use $E(F)[N]^*$ to denote the set of $F$-rational points on $E$ of exact order $N$.

Finally, given a prime $\ell\in\Z^+$ we will use $v_\ell\colon \Q\rightarrow \Z\cup\lbrace \infty\rbrace$ to denote the usual $\ell$-adic valuation.
\section{Strong Uniform $\ell$-Adic Divisibilities for Fields of Definition of Cyclic $\ell$-Primary Isogenies}
Before we prove that $\cI_{F_0}$ and $\cI_{d_0}$ are typically bounded in torsion, let us develop some necessary theory for fields of definition of cyclic isogenies.
Let $F$ be a number field and $E_{/F}$ an elliptic curve. Then $E$ has no (geometric) CM iff there exists a prime $\ell\in\Z^+$ for which the $\ell$-adic representation $\rho_{E,\ell^\infty}(G_F)$ is open in $\GL_2(\Z_\ell)$ \cite[Theorem IV.2.2]{Ser98}. In such a case, $\rho_{E,\ell^\infty}(G_F)$ contains the kernel of the mod-$\ell^n$ reduction map $\GL_2(\Z_\ell)\rightarrow \GL_2(\Z/\ell^n\Z)$ for sufficiently large $n\in\Z^+$. Let us call the least such $n$ the \textit{level} of $\rho_{E,\ell^\infty}(G_F)$.

The following result is due to Cremona and Najman \cite{CN}.
\begin{proposition}\cite[Proposition 3.7]{CN}\label{CyclicIsogFieldDegreeIsEll}
If $E_{/F}$ is a non-CM elliptic curve defined over a number field and $\ell\in\Z^+$ is a prime for which the $\ell$-adic representation $\rho_{E,\ell^\infty}(G_F)$ has level $N$, then for all $n> N$ and for any cyclic subgroup $C\subseteq E(\overline{\Q})$ of order $\ell^{n}$ one has 
\[
[F(C):F(\ell C)]=\ell.
\]
\end{proposition}
The following proposition is a generalization of Proposition \ref{CyclicIsogFieldDegreeIsEll}, which we will prove in this section and apply in the next.
\begin{proposition}\label{DegreeDivFromCyclicIsogeny}
Fix an integer $d_0\in\Z^+$ and a prime $\ell\in\Z^+$. Then there exists an integer $A(d_0,\ell)\in\Z^+$ for which the following holds: for all integers $n\in\Z^+$ with $n\geq A(d_0,\ell)$, and for all non-CM elliptic curves $E_{/F}$ where $[F:\Q]=d_0$, one has for any cyclic subgroup $C\subseteq E(\overline{\Q})$ of order $\ell^n$  the equality
\[
[F(C):F(\ell^{n-A(d_0,\ell)}C)]=\ell^{n-A(d_0,\ell)}.
\]
In particular, one has the divisibility
\[
\ell^{n-A(d_0,\ell)}\mid [F(C):F].
\]
\end{proposition}
Our strengthening uses the following uniformity result on levels of $\ell$-adic representations of non-CM elliptic curves, first proven by Arai \cite[Theorem 1.2]{Ara08} and later strengthened by Clark and Pollack \cite[Theorem 2.3.a]{CP18}. 
\begin{theorem}\cite[Theorem 2.3.a]{CP18}\label{TheoremStrongArai}
Fix an integer $d_0\in \Z^+$. Then for each prime $\ell\in\Z^+$ there exists an integer $A(d_0,\ell)\in\Z^+$ such that for all number fields $F$ with $[F:\Q]=  d_0$ and for all non-CM elliptic curves $E_{/F}$, the $\ell$-adic representation $\rho_{E,\ell^\infty}(G_F)$ has level at most $A(d_0,\ell)$.
\end{theorem}
\begin{proof}[Proof of Proposition \ref{DegreeDivFromCyclicIsogeny}]
Suppose that $E_{/F}$ is a non-CM elliptic curve with $[F:\Q]= d_0$, and suppose that $\ell\in\Z^+$ is a prime for which $\rho_{E,\ell^\infty}(G_F)$ has level $N$; then by Theorem \ref{TheoremStrongArai} we have $N\leq A(d_0, \ell)$.

Fix $n>N$ and a cyclic subgroup $C\subseteq E(\overline{\Q})$ of order $\ell^n$. Then Proposition \ref{CyclicIsogFieldDegreeIsEll} implies that for each $0\leq k<n-N$, the cyclic $\ell^{n-k}$-isogeny $\ell^{k}C\subseteq E(\overline{\Q})$ is such that
\[
[F(\ell^{k}C):F(\ell^{k+1}C)]=\ell.
\]
Since $N\leq A(d_0,\ell)$, from the tower of degree $\ell$ extensions
\[
F(C)\supsetneq F(\ell C)\supsetneq  \ldots \supsetneq  F(\ell^{n-N-1}C)\supsetneq  F(\ell^{n-N}C)
\]
our results immediately follow.
\end{proof}
\section{Typically bounding torsion on $\cI_{F_0}$ and $\cI_{d_0}$}\label{SectionTBT}
To begin with, let us recall from \cite{CMP18} properties \textbf{P1} and \textbf{P2} which apply to certain families $\cF$ of elliptic curves.
\begin{enumerate}[start=1,label={\bfseries P1:}]
\item Given integers $\ell,n_0\in\Z^+$ with $\ell$ prime, there exists $n:=n(\cF, \ell,n_0)\in\Z^+$ such that for all $E_{/F}\in \cF$, if $E(F)[\ell^n]^*\neq\emptyset$ then 
\[
\ell^{n_0}\mid [F:\Q].
\]
\end{enumerate} 
\begin{enumerate}[start=1,label={\bfseries P2:}]
\item There exists $c:=c(\cF)\in \Z^+$ such that for all primes $\ell\in\Z^+$ and all $E_{/F}\in \cF$, if $E(F)[\ell]^*\neq \emptyset$ then 
\[
\ell-1\mid c[F:\Q].
\]
\end{enumerate}
\begin{remark}\label{RemarkSufficientlyLargeP2}
For a family $\cF$, if for some constants $c,\ell_0\in\Z^+$ we show that for all $E_{/F}\in \cF$ and for all primes $\ell\geq \ell_0$ with $E(F)[\ell]^*\neq\emptyset$ one has $\ell-1\mid c[F:\Q]$,
then it follows that $\cF$ satisfies \textbf{P2} with the constant $c(\cF):=c\cdot \prod_{\ell\leq \ell_0} (\ell-1)$. In particular, to prove that $\cF$ satisfies \textbf{P2,} we only need to show that $\ell-1\mid c[F:\Q]$ holds for sufficiently large primes $\ell$ with respect to $\cF$.
\end{remark}
The utility of property \textbf{P1} and \textbf{P2} is in the following theorem.
\begin{theorem}\cite[Theorem 3.2]{CP18}
If a family $\cF$ satisfies \textbf{P1} and \textbf{P2,} then $\cF$ is typically bounded in torsion.
\end{theorem}
As a note, property \textbf{P1} is based off the fact that over number fields of a fixed degree, degrees of $\ell$-primary torsion points on elliptic curves uniformly tend to infinity as the torsion point orders get larger. Property \textbf{P2} is related to a theorem of Erd\H{o}s and Wagstaff \cite[Theorem 2]{EW80}. For more context, see the proof of \cite[Theorem 3.2]{CMP18}.

In this section, we will show that both $\cI_{F_0}$ and $\cI_{d_0}$ satisfy \textbf{P1}, and that $\cI_{F_0}$ satisfies \textbf{P2} but $\cI_{d_0}$ only does conditionally. Let us note that a finite union of families which are typically bounded in torsion will also be typically bounded in torsion. Therefore, since the family of CM elliptic curves is typically bounded in torsion \cite[Theorem 1.1.i]{BCP17}, we will assume hereafter that all of our elliptic curves are non-CM.

\subsection{$\cI_{F_0}$ and $\cI_{d_0}$ satisfy P1}
\begin{proposition}\label{I_d satisfies P1}
For each $d_0\in\Z^+$, the family $\cI_{d_0}$ satisfies \textbf{P1.}
\end{proposition}
\begin{proof}
Fix an integer $n_0\in\Z^+$ and a prime $\ell\in\Z^+$. 
Let us take $A:=A(d_0,\ell)$ to be the ``strong uniform $\ell$-adic Arai level" from Theorem \ref{TheoremStrongArai}.
Let us set $n:=2(A+n_0+v_\ell(d_0!))+1$; we will show that \textbf{P1} holds for $\cI_{d_0}$ with this integer $n$.

Let $E_{/F}$ be a non-CM elliptic curve isogenous to an elliptic curve with degree $d_0$ $j$-invariant, denoted $j'$. Taking a quadratic twist if necessary, by \cite[Proposition 3.3]{Cla} we have that $E$ is $L$-rationally isogenous to an elliptic curve $E'$ defined over $\Q(j')$ with $j(E')=j'$,
where $L/F(j')$ is at most a quadratic extension. 
Let us write this $L$-rational isogeny as $\phi:E\rightarrow E'$; we may assume that $\phi$ is cyclic.

Let us assume that $E(F)[\ell^n]^*\neq\emptyset$; choose any point $P\in E(F)[\ell^n]^*$. Let us set $k:=v_\ell(|\phi(P)|)$, so that the order $|\phi(P)|=\ell^k$ where $k\geq \max\lbrace n-v_\ell(\deg \phi),0\rbrace$. Then $\phi(P)$ generates an $L$-rational cyclic $\ell^k$-subgroup $C'\subseteq E'(\overline{\Q})$. We consider two cases.
\begin{enumerate}[1.]
\item If $n-v_\ell(\deg \phi )\geq A+n_0+v_\ell(d_0!)+1$: then necessarily $n-v_\ell(\deg\phi)\geq 0$. We get that $k\geq A+n_0+v_\ell(d_0!)+1$, so by applying Proposition \ref{DegreeDivFromCyclicIsogeny} to the $L$-rational cyclic $\ell^{k}$-subgroup $C'$, we have
\[
\ell^{n_0+v_\ell(d_0!)+1}\mid [\Q(j')(C' ):\Q(j')]\mid [L:\Q(j')]\mid 2(d_0)![F:\Q].
\]
\item If $n-v_\ell(\deg\phi)< A+n_0+v_\ell(d_0!)+1$: then $v_\ell(\deg\phi)>n-1-(A+n_0+v_\ell(d_0!))=A+n_0+v_\ell(d_0!)$. The subgroup $C'':=(\deg\phi/\ell^{v_\ell(\deg\phi)})\cdot \ker\phi^\vee$ is an $L$-rational cyclic $\ell^{v_\ell(\deg\phi)}$-subgroup of $E'$,\footnote{Here, $\phi^\vee:E'\rightarrow E$ is the dual isogeny of $\phi:E\rightarrow E'$, see \cite[Chapter III.6]{Sil09}. The dual isogeny is also cyclic, and has the same degree and field of definition as $\phi$.} so because $v_\ell(\deg\phi)\geq A+n_0+v_\ell(d_0!)+1$, Proposition \ref{DegreeDivFromCyclicIsogeny} implies that
\[
\ell^{n_0+v_\ell(d_0!)+1}\mid [\Q(j')(C''):\Q(j')]\mid [L:\Q(j')]\mid 2(d_0)![F:\Q].
\]
\end{enumerate}
We thus find in both cases that $\ell^{n_0}\mid [F:\Q]$, whence we conclude that $\cI_{d_0}$ satisfies \textbf{P1.} 
\end{proof}

\begin{corollary}\label{I_F satisfies P1}
For each number field $F_0$, the family $\cI_{F_0}$ satisfies \textbf{P1.}
\end{corollary}
\begin{proof}
Since $\cI_{d_0}$ satisfies \textbf{P1} for all $d_0\in\Z^+$ by Proposition \ref{I_d satisfies P1}, from the containment 
\[
\cI_{F_0}\subseteq \bigcup_{1\leq d\leq [F_0:\Q]}\cI_d
\]
it immediately follows that $\cI_{F_0}$ also satisfies \textbf{P1}, since this property is  closed under finite unions of families.
\end{proof}
\subsection{$\cI_{F_0}$ satisfies \textbf{P2}, as does $\cI_{d_0}$ conditionally}
Throughout the following, we let $d_0$ denote the degree of $F_0$.
\begin{proposition}\label{I_F satisfies P2}
For each number field $F_0$, the family $\cI_{F_0}$ satisfies \textbf{P2.}
\end{proposition}
\begin{proof}
We must construct a constant $c:=c(\cI_{F_0})\in \Z^+$ such that for any non-CM elliptic curve $E_{/F}$ isogenous to an elliptic curve with $F_0$-rational $j$-invariant, if $\ell\in\Z^+$ is a prime such that $E(F)[\ell]^*\neq\emptyset$, then one has 
\begin{equation}\label{P2Div}
\ell-1\mid c[F:\Q].
\end{equation}

Before doing this, we will explain the constant $c(\cE_{F_0})$ that we get in \cite[Theorem 1]{Gen22}, since it is closely related to the constant $c(\cI_{F_0})$ we will construct. As shown in the proof that the family 
\[
\cE_{F_0}:=\lbrace E_{/F}: j(E)\in F_0\rbrace
\] 
conditionally satisfies \textbf{P2} \cite[Theorem 4.3]{CMP18}, for any non-CM elliptic curve $E_{/F_0}$ and for any prime $\ell\in\Z^+$, at least one of the following holds: one has $\ell-1\mid 2[F_0(R):\Q]$ for all $R\in E[\ell]^*$, or else $\ell\leq 15d_0+1$ or $E$ has an $F_0$-rational $\ell$-isogeny; these correspond to Case 1-3, Case 4 and Case 5 of the proof of \cite[Theorem 4.3]{CMP18}, respectively. For Case 1-3 the authors can take $c:=2$ in \eqref{P2Div}, and by Remark \ref{RemarkSufficientlyLargeP2} they can exclude Case 4. For Case 5, their hypotheses imply that $\ell$ is contained in a \textit{finite} set of primes $S_{F_0}$ from \cite[Theorem 1]{LV14}, which they then exclude by Remark \ref{RemarkSufficientlyLargeP2}. Then \cite[Theorem 1]{Gen22} makes Case 5 unconditional: if $\ell \not\in S_{F_0}$, as well as $\ell>72d_0-1$ and $\ell$ is unramified in any imaginary quadratic order whose whose ring class field has degree at most $d_0$, then one has $\ell-1\mid 864[F_0(R):\Q]$ for all $R\in E[\ell]^*$. By Remark \ref{RemarkSufficientlyLargeP2}, this proves that $\cE_{F_0}$ satisfies \textbf{P2} -- an arbitrary non-CM $E_{/F}\in \cE_{F_0}$, while not necessarily defined over $F_0$, will be isomorphic to an elliptic curve $E'_{/F_0}$ over a quadratic extension of $F$. Note that working up to quadratic twist implies one must take $c(\cE_{F_0}):=2\cdot 864=1728$ for $\ell\gg_{F_0}0$.

For the rest of this proof, let us take $\ell\gg_{F_0}0$ to mean that $\ell\not\in S_{F_0}$, $\ell>72d_0-1$ and that $\ell$ is unramified in all imaginary quadratic orders $\oo$ whose ring class field $K(\oo)$ satisfies $[K(\oo):\Q]\leq d_0$. Then as noted above, we have
for all non-CM elliptic curves $E_{/F}\in \cE_{F_0}$ and for all primes $\ell\gg_{F_0} 0$ that $\ell-1\mid 1728[F:\Q]$ whenever $E(F)[\ell]^*\neq\emptyset$.

Let $E_{/F}\in \cI_{F_0}$ be a non-CM elliptic curve. Then by assumption, $E$ is isogenous to an elliptic curve with $F_0$-rational $j$-invariant, denoted $j'$. Arguing as in the last subsection, by \cite[Proposition 3.3]{Cla} there exists both an elliptic curve $E'$ defined over $\Q(j')$ with $j(E')=j'$ and an $L$-rational isogeny $\phi:E\rightarrow E'$, where $L/F(j')$ is at most a quadratic extension. 

Fix a prime $\ell\gg_{F_0}0$, and suppose that $E(F)[\ell]^*\neq\emptyset$. Then by \cite[Corollary 4.3]{BN} there exists an extension $M/L$ of degree dividing $\ell$ for which $E'(M)[\ell]^*\neq\emptyset$. Since $E'_{/M}\in \cE_{F_0}$, by our previous discussion on $c(\cE_{F_0})$ it follows that
\[
\ell-1\mid 1728[M:\Q].
\] 
We check that
\[
[M:\Q]\mid 2\ell[F(j'):\Q]=2\ell[F(j'):F]\cdot [F:\Q]\mid 2\ell(d_0)!\cdot [F:\Q].
\]
Since $\gcd(\ell-1,\ell)=1$, we deduce that
\[
\ell-1\mid 3456(d_0)!\cdot [F:\Q].
\]
Thus the constant $c:=3456(d_0)!$ is such that \eqref{P2Div} holds for all non-CM elliptic curves $E_{/F}\in \cI_{F_0}$ when $\ell\gg_{F_0} 0$. We conclude by Remark \ref{RemarkSufficientlyLargeP2} that $\cI_{F_0}$ satisfies \textbf{P2.}
\end{proof}
For our final result, recall that \textbf{LV}$(d_0)$ is the assumption that the set $S_{F_0}$ from \cite[Theorem 1]{LV14} can be chosen to be the same between any degree $d_0$ number field $F_0$. 
\begin{corollary}\label{I_d satisfies P2 if LV(d) is true}
For any integer $d_0\in \Z^+$, if $\textbf{\emph{LV}}(d_0)$ is true then the family $\cI_{d_0}$ satisfies \textbf{P2.}
\end{corollary}
\begin{proof}
Our proof of Proposition \ref{I_F satisfies P2}
showed that for any number field $F_0$, for sufficiently large primes $\ell\gg_{F_0}0$ and for all non-CM $E_{/F}\in \cI_{F_0}$, if $E(F)[\ell]^*\neq\emptyset$ then 
\begin{equation}\label{EqP2DivisibilityExplicit}
\ell-1\mid 3456(d_0)!\cdot [F:\Q]
\end{equation}
where $d_0:=[F_0:\Q]$.
Our implied constant $\ell\gg_{F_0}0$ was such that $\ell\not\in S_{F_0}$, $\ell>72d_0-1$ and $\ell$ is unramified in any imaginary quadratic order whose whose ring class field has degree at most $d_0$.
Since we are assuming that \textbf{LV}$(d_0)$ is true, this implied constant $\ell\gg_{F_0} 0$ can be chosen to be the same between any number field of degree $d_0$; let us write this as $\ell\gg_{d_0} 0$. In particular, for any non-CM $E_{/F}\in \cI_{d_0}$, fixing a degree $d_0$ $j$-invariant $j'$ which is isogenous to $E$, we have $E_{/F}\in \cI_{F_0}$ where $F_0:=\Q(j')$, and so our proof of Proposition \ref{I_F satisfies P2} shows that \eqref{EqP2DivisibilityExplicit} holds when $\ell\gg_{d_0} 0$ and $E(F)[\ell]^*\neq\emptyset$. Since $E_{/F}\in \cI_{d_0}$ was arbitrary and the implied constant $\ell\gg_{d_0}0$ depends only on $d_0$, we conclude by Remark \ref{RemarkSufficientlyLargeP2} that $\cI_{d_0}$ satisfies \textbf{P2.}
\end{proof}
\subsection{Acknowledgments}
The author thanks Pete L. Clark for helpful comments on earlier drafts of this paper. The author also thanks the referees for their comments and suggestions.

\end{document}